\documentclass[12pt]{article}

\usepackage{amsmath,amsfonts,amssymb,amsthm, color}
\usepackage{comment}
\usepackage{multirow}

\usepackage{float}

\usepackage{caption}

\usepackage{graphicx}

\usepackage[T1]{fontenc}

\newcommand\blfootnote[1]{
  \begingroup
  \renewcommand\thefootnote{}\footnote{#1}
  \addtocounter{footnote}{-1}
  \endgroup
}

\newtheorem{theorem}{Theorem}
\newtheorem{corollary}{Corollary}
\newtheorem{proposition}{Proposition}
\newtheorem{lemma}{Lemma}

\begin{document}

\begin{center}

{\LARGE Asymptotic distributions for weighted power sums of extreme values}

\bigskip
\bigskip
\textsc{Lillian Achola Oluoch} and \textsc{L\'aszl\'o Viharos}

\medskip
Bolyai Institute, University of Szeged \\
Aradi v\'ertan\'uk tere 1, 6720, Szeged, Hungary
\end{center}
\blfootnote{\hspace{-19pt}{\it AMS Subject Classification}: 60F05, 62G32. 
{\it E-mail addresses}: oluoch@math.u-szeged.hu (Lillian Achola Oluoch), viharos@math.u-szeged.hu (L\'aszl\'o Viharos)}

\begin{abstract}
Let $X_{1,n}\le\cdots\le X_{n,n}$ be the
order statistics of $n$ independent random variables with a common
distribution function $F$ having right heavy tail with tail index $\gamma$. Given known constants
$d_{i,n}$, $1\le i\le n$, consider the weighted power sums $\sum^{k_n}_{i=1}d_{n+1-i,n}\log^pX_{n+1-i,n}$, where $p>0$ and the $k_n$ are positive integers such that $k_n\to\infty$ and $k_n/n\to0$ as $n\to\infty$. Under some constraints on the weights $d_{i,n}$, we prove asymptotic normality for the power sums over the whole heavy-tail model. We apply the obtained result to construct a new class of estimators for the parameter $\gamma$.
\end{abstract}

\section{Introduction and results.} Let $X,X_1,X_2,\ldots$ be independent random variables with a common distribution function $F(x)=P\{X\le x\},\ x\in\mathbb R,$ and for each integer $n\ge 1$ let $X_{1,n}\le\cdots\le X_{n,n}$ denote the order statistics pertaining to the sample $X_1,\ldots,X_n$. For a constant $\gamma>0$, let ${\cal R}_\gamma$ be the class of all probability distribution functions $F$ such that
\[
1-F(x)=x^{-1/\gamma} L(x),\quad 0<x<\infty,
\]
where $L$ is a function slowly varying at infinity.
Without loss of generality we assume that $F(1-) = 0$ for all $F \in {\cal R}_\gamma$.
If $Q(\cdot)$ denotes the quantile function of $F$ defined as
\[Q(s)=\inf\{x:F(x)\ge s\},\ 0<s\le 1,\quad Q(0)=Q(0+),\]
then $F \in {\cal R}_\gamma$ if and only if
\begin{equation}
\label{eq:quant}
Q( 1- s) = s^{-\gamma} \ell(s),
\end{equation}
where $\ell$ is a slowly varying function at 0.
Let $k_n$ be a sequence of integers such that
\begin{equation}
\label{kn}
1\le k_n < n, \quad\ k_n\to\infty\quad\ {\rm and}\quad\ k_n/n\to0 \quad{\rm as}\ n\to\infty.
\end{equation}
For some constants $d_{i,n}$, $1\le i\le n$, consider the weighted power sums of the extreme values $X_{n-k_n+1,n},\ldots,X_{n,n}$:
\[
S_n(p):=\sum^{k_n}_{i=1}d_{n+1-i,n}\log^pX_{n+1-i,n},
\]
where $p>0$ is a fixed number. Our aim is to study the asymptotic behavior of $S_n(p)$ as $n\to\infty$ whenever $F \in {\cal R}_\gamma$.

Cs\"org\H o et al.~\cite{CsHM} found necessary and sufficient
conditions for the existence of normalizing and centering constants
$A_n>0$ and $C_n$ such that the sequence
\[\frac1{A_n}\left\{\sum_{i=1}^{k_n}X_{n+1-i,n}-C_n\right\}\]
converges in distribution along subsequences of the integers $\{n\}$ to
non-degenerate limits and completely described the possible
subsequential limiting distributions.
Viharos \cite{Viharos93} generalized this result for linear combinations $\sum_{i=k+1}^{k_n}d_{n+1-i,n}f(X_{n+1-i,n})$ of extreme values, where $f$ is a Borel-measurable function. Assuming $F \in {\cal R}_\gamma$ and using the results in \cite{Viharos93}, we will prove asymptotic normality for the properly normalized and centered sequence $S_n(p)$.
As an application, we derive a class of asymptotically normal estimators for the parameter $\gamma$.

Linear combinations of order statistics are widely studied in the literature.
Recently, Barczyk et al.~\cite{Barczyk} obtained limit theorems for L-statistics
\[
L_n = \sum_{i=1}^{k_n} c_{i,n} X_{i:k_n},
\]
where $k_n\to\infty$ as $n\to\infty$, $c_{i,n}$ are real scores and the order statistics $X_{i:k_n}$ correspond to a possibly non i.i.d.~triangular array $(X_{i,n})_{1\le i\le k_n}$ of infinitesimal and rowwise independent random variables with heavy tails. Their approach is related to the extreme order statistics: they give sufficient conditions for the scores $c_{i,n}$ so that only the extreme parts of the L-statistics contribute to the limit law.

We will assume as in \cite{Viharos93} that the weights $d_{i,n}$ are of the form
\[
d_{i,n}=n\int_{(i-1)/n}^{i/n}\bar L(t)dt,
\quad 1\le i\le n,
\]
for some non-negative continuous function $\bar L$ defined on (0,1)
which satisfies the following condition:

\bigskip\noindent
Condition $\mathbf{\bar L}$:

a) There exists a constant $-1/2<\rho<\infty$ such that $\bar L(1-t)=t^{\rho} \bar\ell(t)$ on $(0,1)$ for some function $\bar\ell(\cdot)$ slowly varying at $0$ and $\bar\ell'(t)=t^{-1}\bar\ell(t)\varepsilon(t)$ on some $(0,\delta)$ with a continuous function $\varepsilon(\cdot)$ for which $\varepsilon(t)\to0$ as $t\to0$.

b) For all $M\ge 1$,
\[\sup_{1/M<y<M} \left|\int\limits_0^y\frac{(\bar\ell(u/n)-\bar\ell(y/n))u^\rho}{\bar\ell(y/n)y^\rho}du\right|\to 0,\quad n\to\infty.\]

Throughout the paper we use the convention
$\int_a^b=\int_{[a,b)}$ when we integrate with respect to
a left continuous integrator. Define
\[
J(s)=s^{\rho} \bar\ell(s),\quad 0<s<1,
\]
and
\[
g(t)=-(\log Q(1-t-))^p,
\]
where $Q(1-s-)$ denotes the left-continuous version of the right-continuous function $Q(1-s)$, $0<s<1$,
\[
K(t)=\int_{1/2}^t J(s)dg(s),\quad 0<t<1,
\]
and
\[
\sigma^2(s,t)=\int^t_s\int^t_s(u\wedge v-uv)dK(u)dK(v),\quad 0\le s\le t\le 1,
\]
where $u\wedge v=\min(u,v)$.
We introduce the centering sequences
\[
\mu_n:=-n\int^{k_n/n}_0J(u)g(u)du,
\]
and 
\[
\bar \mu_n = -n\int^{k_n/n}_{1/n} J(u)g(u) du-d_{n,n}g\!\left(\frac1n\right),
\]
while the normalizing sequence will be given by
\[
a_n:=
\begin{cases}
\sigma(1/n,k_n/n) &\ \text{if}\quad \sigma(1/n,k_n/n)>0,\\
1 &\ \text{otherwise.}
\end{cases}
\]

We state now the main limit theorem of the paper. Throughout, $\stackrel{\mathcal{D}}{\longrightarrow}$ denotes convergence in distribution, $\stackrel{\mathbb{P}}{\longrightarrow}$ denotes convergence in probability, and limiting and order relations are always meant as $n\to\infty$ if not specified otherwise.

\begin{theorem}
\label{thm1}
{\rm (i)} Assume that $F\in{\cal R}_\gamma$, {\rm (\ref{kn})} holds  and suppose that condition $\mathbf{\bar{L}}$ is satisfied for the weighs $d_{i,n}$. Then
\begin{equation}
\label{eq:asnorm}
\frac1{\sqrt n a_n}\left\{\sum^{k_n}_{i=1}d_{n+1-i,n}\log^p X_{n+1-i,n} - \bar \mu_n\right\} \stackrel{\mathcal{D}}{\longrightarrow} N(0,1).
\end{equation}
{\rm (ii)} If in addition to the conditions of {\rm (i)} we have $(\log n)/k_n^\varepsilon\to0$ for some $0 <\varepsilon < \rho+1/2$, then \eqref{eq:asnorm} holds with $\mu_n$ replacing $\bar \mu_n$.
\end{theorem}

The special case $p=1$ of Theorem \ref{thm1}(i) was stated in Theorem 1.2 of \cite{Viharos95}.
Several estimators exist for the tail index $\gamma$ among which Hill's estimator
is the most classical (see Hill \cite{Hill}).
Dekkers et al.~\cite{Dekkers} proposed a moment estimator based on the statistics
\begin{equation}
\label{eq:def-Mn}
\frac{1}{k_n} \sum_{i=1}^{k_n}
\left( \log \frac{X_{n+1-i,n}}{X_{n-k_n,n}} \right)^j,\quad j = 1,2.
\end{equation}
The case $j=1$ yields the Hill estimator.
Segers \cite{Segers} investigated more general statistics of the form
\begin{equation}
\label{eq:fstat}
\frac{1}{k_n} \sum_{i=1}^{k_n} 
f\left( \frac{X_{n+1-i,n}}{X_{n-k_n,n}} \right),
\end{equation}
for a nice class of functions $f$, called residual estimators. Segers proved weak consistency and asymptotic normality under general conditions.
More recent\-ly, Ciuperca and Mercadier \cite{Ciuperca} obtained a class of tail index estimators based on the weighted power sums of the statistics $\big(\log(X_{n+1-i,n} / X_{n-k_n,n})\big)_{1\le i\le k_n}$ and proved limit theorems for the estimators.
We use the weighted power sums of the extreme values $(\log X_{n+1-i,n})_{1\le i\le k_n}$ to construct a new class of estimators for $\gamma$.

The following proposition describes the asymptotic behavior of the centering and normalizing sequences.

\begin{proposition}
\label{propseq}
Assume the conditions of Theorem {\rm \ref{thm1}(i)}. Then
\begin{equation}
\label{sigma}
\sigma(1/n,k_n/n)\sim p\gamma^p\left(\frac2{(1+\rho)(1+2\rho)}\right)^{1/2}\left(\frac{k_n}n\right)^{\rho+1/2}
\left(\log\frac n{k_n}\right)^{p-1}\bar \ell\left(\frac{k_n}n\right)
\end{equation}
and $\mu_n\sim\gamma^p\alpha_n$, where $\alpha_n=\frac{k_n}{\rho+1}J\left(\frac{k_n}n\right)\big(\log\frac n{k_n}\big)^p$
($x_n\sim y_n$ means that $x_n/y_n\to1$).
\end{proposition}

The next corollary describes the asymptotic behavior of the weighted norms $R_n(p):=(S_n(p))^{1/p}$.

\begin{corollary}
\label{cor1}
Assume the conditions of Theorem {\rm \ref{thm1}(ii)}. Then
\[
\frac1{\gamma\sqrt2}\left(\frac{1+2\rho}{1+\rho}\right)^{1/2}\sqrt{k_n}\log\frac n{k_n}\left\{\frac1{\alpha_n^{1/p}}R_n(p)-\left(\frac{\mu_n}{\alpha_n}\right)^{1/p}\right\} \stackrel{\mathcal{D}}{\longrightarrow} N(0,1).
\]
\end{corollary}

By Proposition \ref{propseq} and Corollary \ref{cor1}, 
\[
\widehat\gamma_n:=\frac1{\alpha_n^{1/p}}R_n(p)
\]
is an asymptotically normal estimator for $\gamma$. This is a generalization of the estimator proposed in \cite{Viharos97}.
Asymptotic normality was proved for the Hill estimator and for the estimators in \cite{Ciuperca} and \cite{Segers} under general conditions but not for every distribution in ${\cal R}_\gamma$. However, $\widehat\gamma_n$ is asymptotically normal over the whole model ${\cal R}_\gamma$.

To investigate the asymptotic bias of the estimator $\widehat\gamma_n$, we assume the following conditions:

\bigskip\noindent
$\displaystyle
(B_1)\quad
\sqrt{k_n}\log\frac n{k_n} \sup_{0\le u \le k_n/n}\left|\frac{\log\ell(u)}{\log u}\right| \to0.$

\smallskip\noindent
($B_2$) $\sqrt{k_n}/\log n \to0$.

\smallskip\noindent
($B_3$) $(\log n)/ k_n^{\rho+\frac12} n \to0$.

\smallskip\noindent
($B_4$) $J(s)=s^{\rho},\quad 0<s<1$.

\bigskip
Conditions ($B_2$) and ($B_3$) imply that $\rho>0$.

\begin{corollary}
\label{cor2}
Assume the conditions {\rm ($B_1$)-($B_4$)}, and the conditions of Theorem {\rm \ref{thm1}(i)}, and set $t_n:=(\rho+1)\log(n/k_n)$. Then we have

{\rm (i)}
\[
\frac1{\gamma^p p\sqrt2}\left(\frac{1+2\rho}{1+\rho}\right)^{1/2}\sqrt{k_n}\log\frac n{k_n}
 \left\{\frac{S_n(p)}{\alpha_n} - \gamma^p\big(1 + p t_n^{-1}\big)\right\}
 \stackrel{\mathcal{D}}{\longrightarrow} N(0,1),
\]

{\rm (ii)}
\begin{equation}
\label{eq:norm1}
\frac1{\gamma\sqrt2}\left(\frac{1+2\rho}{1+\rho}\right)^{1/2}\sqrt{k_n}\log\frac n{k_n}
\big\{\widehat\gamma_n- \gamma\big(1+ t_n^{-1}\big)\big\}
 \stackrel{\mathcal{D}}{\longrightarrow} N(0,1).
\end{equation}
\end{corollary}

We show that condition ($B_1$) is satisfied by the model $\ell(s) = 1 + b(s)$, where the function $b$ is such that $\sqrt{k_n}\sup_{0\le u \le k_n/n}|b(u)| \to0$. To prove this, observe that  $\sup_{0\le u \le k_n/n}
1/|\log u| = 1/\log(n/k_n)$ and hence
\[
\begin{split}
\sqrt{k_n}&\log\frac n{k_n} \sup_{0\le u \le k_n/n}\left|\frac{\log\ell(u)}{\log u}\right|
\le \sqrt{k_n} \sup_{0\le u \le k_n/n} |\log(1+b(u)| \\
& =  \sqrt{k_n} \sup_{0\le u \le k_n/n} |b(s) + O(b^2(s))|
\to0,
\end{split}
\]
if $\sqrt{k_n}\sup_{0\le u \le k_n/n}|b(u)| \to0$.

In some submodels of \eqref{eq:quant} the Hill estimator can be centered at $\gamma$ to have normal asymptotic distribution. The strict Pareto model when $\ell \equiv 1$ is the simplest example of these models. This simple  model satisfies the conditions of Corollary \ref{cor2}.
From \eqref{eq:norm1} we also see that under these conditions the estimator $\widehat\gamma_n$ can not be centered at $\gamma$ to have asymptotic distribution. However, Corollary \ref{cor2} allows the construction of asymptotic confidence intervals for
$\gamma$. The estimator $\widehat\gamma_n$ is not scale invariant. Accordingly, the slowly varying function $\ell \equiv c$, $c \ne 1$, does not satisfy condition ($B_1$).

\section{Simulation results}
\label{simulation}
In this section we evaluate the performance of the estimator $\widehat\gamma_n$ through simulations.
In the first simulation study we compare $\widehat\gamma_n$ to the Hill, Pickands (\cite{Pickands}) and moment estimators.
Tail index estimators have good performance in the strict Pareto model.
However, in practical situations it is very rare when data fit to a simple distribution. For the simulation we use the following model proposed by Hall \cite{Hall}:
\begin{equation}
\label{eq:Hall}
Q(1-s)=s^{-\gamma}D_1[1+D_2s^{\beta}(1+o(1))]\ \quad{\rm as}\ s\to0,
\end{equation}
where $D_1>0$, $D_2\not=0$ and $\beta>0$ are constants. The Hall model satisfies condition ($B_1$) if $D_1=1$ and $k_n^{\beta + \frac12}/n^\beta \to 0$.

We repeated the simulations 1000 times and we assumed $n=1000$ for the sample size and $k_n=136$ for the sample fraction size.
We used $\bar\ell \equiv 1$ for the weights $d_{i,n}$.
We examined the following two cases of the Hall model:

\medskip
{\it Case 1}: $\beta = 2$, $D_2 =1$ and $D_1 = 1/ \sqrt{e}$.

{\it Case 2}: $\beta = 1$, $D_2 = 4/3$ and $D_1 = \,e^{-2/3}$.

\medskip
In both cases we assume $o(1) \equiv 0$ in \eqref{eq:Hall}.
Tables \ref{table_mean1} and \ref{table_mse1} contain the average simulated estimates (mean) and the calculated empirical mean square errors (MSE) for {\it Case 1}.
Using the mean square error as criterion, we see that for $\rho \le 1$ the performance of $\widehat\gamma_n$ generally increases
as $\gamma$ decreases from 2 to 0.5. For $\gamma \ge 1$ the weights improve the performance of $\widehat\gamma_n$ significantly ($\rho = 0.5, 1, 2$).
For the thin tail pertaining to $\gamma = 0.5$ we also see a trend that the performance of $\widehat\gamma_n$ improves as the value of $p$ increases from 1 to 3. The same conclusion holds for $\gamma = 1$ when $\rho = 2$.
It can be also seen that $\widehat\gamma_n$ with $p=1, 2, 3$ and appropriate $\rho$ value performs better than the Pickands and the moment estimator. The Pickands estimator has poor performance for $\gamma = 2$. Nonetheless, the Hill and the moment estimator tend to have good estimates.

Tables \ref{table_mean2} and \ref{table_mse2} contain the simulation results for {\it Case 2}. This case is farther from the strict Pareto model than {\it Case 1}.
In {\it Case 2} for $\rho \le 0.5$ the estimator $\widehat\gamma_n$ works slightly better than in the first case.
The performance of the Hill estimator is slightly worse in this case, while the other estimators have similar performance compared to the first case.

	\begin{table}[H]
		\caption{Mean in the Hall model for {\it Case 1}.}
		\label{table_mean1}
		\scalebox{0.8}{
			\begin{tabular}{ccccclllcc}	
\cline{1-8}			\multicolumn{8}{c}{mean}                                                                                                                                                                                                                                 \\ \cline{1-8} 
		 & &\multicolumn{3}{c|}{$\widehat\gamma_n$}                            & \multicolumn{1}{c|}{\multirow{2}{*}{Hill}}                         & \multicolumn{1}{c|}{\multirow{2}{*}{Pickands}} & \multicolumn{1}{c|}{\multirow{2}{*}{moment}} & 
				 \\ \cline{1-5}
	\multicolumn{1}{c|}{$ \rho $} &	\multicolumn{1}{c|}{$\gamma$} & $p=1$      & $p=2$     & \multicolumn{1}{c|}{$p=3$}   & &    & \multicolumn{1}{c|}{}                      &                        \\
\cline{1-8} 
\multicolumn{1}{c|}{0}& \multicolumn{1}{c|}{0.5}   & 0.502461 & 0.5598067 & \multicolumn{1}{c|}{0.6278012} & \multicolumn{1}{c|}{0.4874154}  & \multicolumn{1}{c|}{0.5388793}& \multicolumn{1}{c|}{0.4832535}                                                                 \\

\multicolumn{1}{c|}{} & \multicolumn{1}{c|}{1}  & 1.252406 & 1.347012

 & \multicolumn{1}{c|}{1.461455} &\multicolumn{1}{l|}{0.9872326}  &\multicolumn{1}{c|}{1.021725}              & \multicolumn{1}{c|}{0.9745838}                               \\
 
 \multicolumn{1}{c|}{}  & \multicolumn{1}{c|}{1.5}   &2.002351  & 2.136447 & \multicolumn{1}{c|}{2.299039}  & \multicolumn{1}{c|}{1.48705}  & \multicolumn{1}{c|}{1.52004}              & \multicolumn{1}{c|}{1.471576}                                \\
 
 \multicolumn{1}{c|}{} & \multicolumn{1}{c|}{2}  & 2.752296 & 2.926308
 
 & \multicolumn{1}{c|}{3.137432} & \multicolumn{1}{c|}{1.986867}  & \multicolumn{1}{c|}{2.022467}              & \multicolumn{1}{c|}{1.969981}                               \\
\cline{1-8} 
 
 \multicolumn{1}{c|}{0.5}&\multicolumn{1}{c|}{0.5}    & 0.4207121 &0.4523482 & \multicolumn{1}{c|}{0.4918764} & \multicolumn{1}{c|}{0.4874154}  & \multicolumn{1}{c|}{0.5388793}              & \multicolumn{1}{c|}{0.4832535}                              \\
 \multicolumn{1}{c|}{}&\multicolumn{1}{c|}{1}     & 1.088022 & 1.138332 & \multicolumn{1}{c|}{1.200928} &\multicolumn{1}{l|}{0.9872326}  &\multicolumn{1}{c|}{1.021725}              & \multicolumn{1}{c|}{0.9745838}                                   \\
 
 \multicolumn{1}{c|}{}&\multicolumn{1}{c|}{1.5}   & 1.755332 & 1.826024 & \multicolumn{1}{c|}{1.913608}  & \multicolumn{1}{c|}{1.48705}  & \multicolumn{1}{c|}{1.52004}              & \multicolumn{1}{c|}{1.471576}                                 \\
 
 \multicolumn{1}{c|}{}&\multicolumn{1}{c|}{2}     & 2.422641 & 2.514022 & \multicolumn{1}{c|}{2.626971} & \multicolumn{1}{c|}{1.986867}  & \multicolumn{1}{c|}{2.022467}              & \multicolumn{1}{c|}{1.969981}                                   \\
\cline{1-8} 
 \multicolumn{1}{c|}{1}&\multicolumn{1}{c|}{0.5}     & 0.37965551 & 0.3994002 & \multicolumn{1}{c|}{0.4240878} &\multicolumn{1}{c|}{0.4874154}  & \multicolumn{1}{c|}{0.5388793}              & \multicolumn{1}{c|}{0.4832535}                               \\
 
 \multicolumn{1}{c|}{}&\multicolumn{1}{c|}{1}      & 1.005246 & 1.03595 & \multicolumn{1}{c|}{1.073641} & \multicolumn{1}{l|}{0.9872326}  &\multicolumn{1}{c|}{1.021725}              & \multicolumn{1}{c|}{0.9745838}                                    \\
 
 \multicolumn{1}{c|}{}& \multicolumn{1}{c|}{1.5}       &1.630837  & 1.673773 & \multicolumn{1}{c|}{1.726098} & \multicolumn{1}{c|}{1.48705}  & \multicolumn{1}{c|}{1.52004}              & \multicolumn{1}{c|}{1.471576}                           \\
 
 \multicolumn{1}{c|}{}& \multicolumn{1}{c|}{2}      & 2.256427 & 2.311814 & \multicolumn{1}{c|}{2.379069} & \multicolumn{1}{c|}{1.986867}  & \multicolumn{1}{c|}{2.022467}              & \multicolumn{1}{c|}{1.969981}                                   \\
\cline{1-8} 
 \multicolumn{1}{c|}{2} &\multicolumn{1}{c|}{0.5}   & 0.33886111 & 0.3486395  & \multicolumn{1}{c|}{0.3606289} & \multicolumn{1}{c|}{0.4874154}  & \multicolumn{1}{c|}{0.5388793}              & \multicolumn{1}{c|}{0.4832535}                                  \\
 
 \multicolumn{1}{c|}{} &\multicolumn{1}{c|}{1}   & 0.9227323 & 0.9375759 & \multicolumn{1}{c|}{0.9552161} & \multicolumn{1}{l|}{0.9872326}  &\multicolumn{1}{c|}{1.021725}              & \multicolumn{1}{c|}{0.9745838}                                \\
 
 \multicolumn{1}{c|}{}&\multicolumn{1}{c|}{1.5}    &1.506604 &1.527265  & \multicolumn{1}{c|}{1.551595} & \multicolumn{1}{c|}{1.48705}  & \multicolumn{1}{c|}{1.52004}              & \multicolumn{1}{c|}{1.471576}   \\
 
 \multicolumn{1}{c|}{}&\multicolumn{1}{c|}{2}    & 2.090475  & 2.117078 & \multicolumn{1}{c|}{2.148269} & \multicolumn{1}{c|}{1.986867}  & \multicolumn{1}{c|}{2.022467}              & \multicolumn{1}{c|}{1.969981}                                             \\
 
\cline{1-8} 
			\end{tabular}}
		\end{table}

	\begin{table}[H]
\vspace{10pt}
				\caption{MSE in the Hall model for {\it Case 1}.}
				\label{table_mse1}
				\scalebox{0.75}{
						\begin{tabular}{ccccclllcc}	
\cline{1-8}						\multicolumn{8}{c}{MSE}                                                                                                                                                                                                                                 \\ \cline{1-8} 
		& &\multicolumn{3}{c|}{$\widehat\gamma_n$}                            & \multicolumn{1}{c|}{\multirow{2}{*}{Hill}}                         & \multicolumn{1}{c|}{\multirow{2}{*}{Pickands}} & \multicolumn{1}{c|}{\multirow{2}{*}{moment}} & 
							\\ \cline{1-5}
	\multicolumn{1}{c|}{$ \rho $} &	\multicolumn{1}{c|}{$\gamma$} & $p=1$      & $p=2$     & \multicolumn{1}{c|}{$p=3$}   & &    & \multicolumn{1}{c|}{}                      &                      
							\\
\cline{1-8}    
\multicolumn{1}{c|}{0}& \multicolumn{1}{c|}{0.5}   & 0.008489717 & 0.004758226 & \multicolumn{1}{c|}{0.01848487} & \multicolumn{1}{c|}{0.001920372}  & \multicolumn{1}{c|}{0.1238975}& \multicolumn{1}{c|}{0.008732585}                                                                 \\ \multicolumn{1}{c|}{} & \multicolumn{1}{c|}{1}  & 0.06713682 & 0.124994 & \multicolumn{1}{c|}{0.2205786} & \multicolumn{1}{l|}{0.007254561}   & \multicolumn{1}{c|}{0.1510138}              & \multicolumn{1}{c|}{0.01456819}                               \\
						
\multicolumn{1}{c|}{}  & \multicolumn{1}{c|}{1.5}   &0.2601122  &0.415274 & \multicolumn{1}{c|}{0.6550551}  &  \multicolumn{1}{c|}{0.01616043}  & \multicolumn{1}{c|}{0.191689}              & \multicolumn{1}{c|}{0.02365229}                                \\       
 \multicolumn{1}{c|}{} & \multicolumn{1}{c|}{2}  & 0.579775 & 0.8761246 & \multicolumn{1}{c|}{1.322759} & \multicolumn{1}{c|}{0.02863798} &  \multicolumn{1}{c|}{0.2457045}              & \multicolumn{1}{c|}{0.0362088}                               \\
\cline{1-8} 
\multicolumn{1}{c|}{0.5}&\multicolumn{1}{c|}{0.5}    & 0.006915487 &0.002963965 & \multicolumn{1}{c|}{0.0009153005}& \multicolumn{1}{c|}{0.001920372}  & \multicolumn{1}{c|}{0.1238975}& \multicolumn{1}{c|}{0.008732585}                             \\
\multicolumn{1}{c|}{}&\multicolumn{1}{c|}{1}     & 0.01033168 & 0.02195434 & \multicolumn{1}{c|}{0.04367552}& \multicolumn{1}{c|}{0.007254561}   & \multicolumn{1}{c|}{0.1510138}              & \multicolumn{1}{c|}{0.01456819}                               \\
\multicolumn{1}{c|}{}&\multicolumn{1}{c|}{1.5}   & 0.07105951 & 0.1126648 & \multicolumn{1}{c|}{0.1784538} & \multicolumn{1}{c|}{0.01616043}  & \multicolumn{1}{c|}{0.191689}              & \multicolumn{1}{c|}{0.02365229}                              \\
\multicolumn{1}{c|}{}&\multicolumn{1}{c|}{2}     & 0.189099 & 0.2755773 & \multicolumn{1}{c|}{0.4061787} & \multicolumn{1}{c|}{0.02863798} &  \multicolumn{1}{c|}{0.2457045}              & \multicolumn{1}{c|}{0.0362088}                                \\
\cline{1-8} 
\multicolumn{1}{c|}{1}&\multicolumn{1}{c|}{0.5} & 0.01503467 & 0.01069469 & \multicolumn{1}{c|}{0.006382895} &\multicolumn{1}{c|}{0.001920372}  & \multicolumn{1}{c|}{0.1238975}& \multicolumn{1}{c|}{0.008732585}                                       \\	\multicolumn{1}{c|}{}&\multicolumn{1}{c|}{1}      & 0.002311005 & 0.003667682 & \multicolumn{1}{c|}{0.007952766} & \multicolumn{1}{c|}{0.007254561}   & \multicolumn{1}{c|}{0.15101382}              & \multicolumn{1}{c|}{0.01456819}                                 \\
\multicolumn{1}{c|}{}& \multicolumn{1}{c|}{1.5}       &0.02231372  & 0.03559411 & \multicolumn{1}{c|}{0.05684494} & \multicolumn{1}{c|}{0.01616043}  & \multicolumn{1}{c|}{0.191689}              & \multicolumn{1}{c|}{0.02365229}                             \\
\multicolumn{1}{c|}{}& \multicolumn{1}{c|}{2}      & 0.07504283 & 0.1068695 & \multicolumn{1}{c|}{0.1538997} & \multicolumn{1}{c|}{0.02863798} &  \multicolumn{1}{c|}{0.2457045}              & \multicolumn{1}{c|}{0.0362088}                               \\
\cline{1-8} 
\multicolumn{1}{c|}{2} &\multicolumn{1}{c|}{0.5}   & 0.02645074 & 0.02340072  & \multicolumn{1}{c|}{0.01992387} & \multicolumn{1}{c|}{0.001920372}  & \multicolumn{1}{c|}{0.1238975}& \multicolumn{1}{c|}{0.008732585}                                \\
\multicolumn{1}{c|}{} &\multicolumn{1}{c|}{1}   & 0.007996087 & 0.005951954 & \multicolumn{1}{c|}{0.004100054} & \multicolumn{1}{c|}{0.007254561}   & \multicolumn{1}{c|}{0.1510138}              & \multicolumn{1}{c|}{0.01456819}                                \\
\multicolumn{1}{c|}{}&\multicolumn{1}{c|}{1.5}    &0.004666964 &0.005432634  & \multicolumn{1}{c|}{0.007437678} & \multicolumn{1}{c|}{0.01616043}  & \multicolumn{1}{c|}{0.191689}              & \multicolumn{1}{c|}{0.02365229}   \\
\multicolumn{1}{c|}{}&\multicolumn{1}{c|}{2}    & 0.01646337& 0.02210106 & \multicolumn{1}{c|}{0.03052874} & \multicolumn{1}{c|}{0.02863798} &  \multicolumn{1}{c|}{0.2457045}              & \multicolumn{1}{c|}{0.0362088}                                              \\
\cline{1-8} 
					\end{tabular}}
				\end{table}

\goodbreak
	\begin{table}[H]
		\caption{Mean in the Hall model for {\it Case 2}.}
		\label{table_mean2}
		\scalebox{0.8}{
			\begin{tabular}{ccccclllcc}	
\cline{1-8}				\multicolumn{8}{c}{mean}                                                                                                                                                                                                                                 \\ \cline{1-8} 
				& &\multicolumn{3}{c|}{$\widehat\gamma_n$}                            & \multicolumn{1}{c|}{\multirow{2}{*}{Hill}}                         & \multicolumn{1}{c|}{\multirow{2}{*}{Pickands}} & \multicolumn{1}{c|}{\multirow{2}{*}{moment}} & 
				\\ \cline{1-5}
				\multicolumn{1}{c|}{$ \rho $} &	\multicolumn{1}{c|}{$\gamma$} & $p=1$      & $p=2$     & \multicolumn{1}{c|}{$p=3$}   & &    & \multicolumn{1}{c|}{}                      &                       \\
\cline{1-8} 
				
				\multicolumn{1}{c|}{0}& \multicolumn{1}{c|}{0.5}   & 0.4589447 & 0.5124346 & \multicolumn{1}{c|}{0.5782619} & \multicolumn{1}{c|}{0.4184847 }  & \multicolumn{1}{c|}{0.609304}& \multicolumn{1}{c|}{0.4811199}                                                   \\
				\multicolumn{1}{c|}{} & \multicolumn{1}{c|}{1}  & 1.20889 & 1.299565
				& \multicolumn{1}{c|}{1.411133} & \multicolumn{1}{c|}{0.9183019}  & \multicolumn{1}{c|}{1.011627}              & \multicolumn{1}{c|}{0.9377913}                               \\
				
				\multicolumn{1}{c|}{}  & \multicolumn{1}{c|}{1.5}   &1.958834 & 2.089031 & \multicolumn{1}{c|}{2.248614}  &\multicolumn{1}{c|}{1.418119}  
				& \multicolumn{1}{c|}{1.488072}              & \multicolumn{1}{c|}{1.423793}                                \\
				\multicolumn{1}{c|}{} & \multicolumn{1}{c|}{2}  & 2.708779 & 2.878913
				
				& \multicolumn{1}{c|}{3.086975} &\multicolumn{1}{c|}{1.917936}  &  \multicolumn{1}{c|}{1.980745}              & \multicolumn{1}{c|}{1.916807}                               \\
\cline{1-8} 
				
				\multicolumn{1}{c|}{0.5}&\multicolumn{1}{c|}{0.5}    & 0.3846648 &0.4124053 & \multicolumn{1}{c|}{0.4486737} & \multicolumn{1}{c|}{0.4184847 }  & \multicolumn{1}{c|}{0.609304}& \multicolumn{1}{c|}{0.4811199}                               \\
				
				\multicolumn{1}{c|}{}&\multicolumn{1}{c|}{1}     & 1.051975 & 1.09872 & \multicolumn{1}{c|}{1.157892}& \multicolumn{1}{c|}{0.9183019}  & \multicolumn{1}{c|}{1.011627}              & \multicolumn{1}{c|}{0.9377913}                               \\
				\multicolumn{1}{c|}{}&\multicolumn{1}{c|}{1.5}   & 1.719284 & 1.786526 & \multicolumn{1}{c|}{1.870711} &\multicolumn{1}{c|}{1.418119}  
				& \multicolumn{1}{c|}{1.488072}              & \multicolumn{1}{c|}{1.423793}                                \\
				\multicolumn{1}{c|}{}&\multicolumn{1}{c|}{2}     & 2.386594 & 2.47458 & \multicolumn{1}{c|}{2.58415} &\multicolumn{1}{c|}{1.917936}  &  \multicolumn{1}{c|}{1.980745}              & \multicolumn{1}{c|}{1.916807}                               \\
\cline{1-8} 
				\multicolumn{1}{c|}{1}&\multicolumn{1}{c|}{0.5}     & 0.3484573 & 0.36485 & \multicolumn{1}{c|}{0.3861637} & \multicolumn{1}{c|}{0.4184847 }  & \multicolumn{1}{c|}{0.609304}& \multicolumn{1}{c|}{0.4811199}                              \\
				\multicolumn{1}{c|}{}&\multicolumn{1}{c|}{1}      & 0.9740478 & 1.001849 & \multicolumn{1}{c|}{1.036418} & \multicolumn{1}{c|}{0.9183019}  & \multicolumn{1}{c|}{1.011627}              & \multicolumn{1}{c|}{0.9377913}                               \\
				\multicolumn{1}{c|}{}& \multicolumn{1}{c|}{1.5}       &1.599638  & 1.639798 & \multicolumn{1}{c|}{1.689107} &\multicolumn{1}{c|}{1.418119}  
				& \multicolumn{1}{c|}{1.488072}              & \multicolumn{1}{c|}{1.423793}                                \\
				\multicolumn{1}{c|}{}& \multicolumn{1}{c|}{2}      & 2.225229 & 2.277898 & \multicolumn{1}{c|}{2.34219}  &\multicolumn{1}{c|}{1.917936}  &  \multicolumn{1}{c|}{1.980745}              & \multicolumn{1}{c|}{1.916807}                               \\
\cline{1-8} 
				\multicolumn{1}{c|}{2} &\multicolumn{1}{c|}{0.5}   & 0.3135496 & 0.3210441  & \multicolumn{1}{c|}{0.3304326} & \multicolumn{1}{c|}{0.4184847 }  & \multicolumn{1}{c|}{0.609304}& \multicolumn{1}{c|}{0.4811199}                            \\
				\multicolumn{1}{c|}{} &\multicolumn{1}{c|}{1}   & 0.8974208 & 0.9103876 & \multicolumn{1}{c|}{0.9258765} & \multicolumn{1}{c|}{0.9183019}  & \multicolumn{1}{c|}{1.011627}              & \multicolumn{1}{c|}{0.9377913}                               \\
				\multicolumn{1}{c|}{}&\multicolumn{1}{c|}{1.5}    &1.481292 &1.500177  & \multicolumn{1}{c|}{1.522475} &\multicolumn{1}{c|}{1.418119}  
				& \multicolumn{1}{c|}{1.488072}              & \multicolumn{1}{c|}{1.423793}                                \\
				
				\multicolumn{1}{c|}{}&\multicolumn{1}{c|}{2}    & 2.065163  & 2.090035 & \multicolumn{1}{c|}{2.119249} &\multicolumn{1}{c|}{1.917936}  &  \multicolumn{1}{c|}{1.980745}              & \multicolumn{1}{c|}{1.916807}                               \\
\cline{1-8} 
		\end{tabular}}
	\end{table}

	\begin{table}[H]
\vspace{10pt}
		\caption{MSE in the Hall model for {\it Case 2}.}
		\label{table_mse2}
		\scalebox{0.75}{
			\begin{tabular}{ccccclllcc}	
\cline{1-8}				\multicolumn{8}{c}{MSE}                                                                                                                                                                                                                                 \\ \cline{1-8} 
				& &\multicolumn{3}{c|}{$\widehat\gamma_n$}                            & \multicolumn{1}{c|}{\multirow{2}{*}{Hill}}                         & \multicolumn{1}{c|}{\multirow{2}{*}{Pickands}} & \multicolumn{1}{c|}{\multirow{2}{*}{moment}} & 
				\\ \cline{1-5}
				\multicolumn{1}{c|}{$ \rho $} &	\multicolumn{1}{c|}{$\gamma$} & $p=1$      & $p=2$     & \multicolumn{1}{c|}{$p=3$}   & &    & \multicolumn{1}{c|}{}                      &                       \\ 
\cline{1-8} 
				\multicolumn{1}{c|}{0}& \multicolumn{1}{c|}{0.5}   &0.002356798 & 0.001205257 & \multicolumn{1}{c|}{0.008251327} & \multicolumn{1}{c|}{0.0081828} & 
				\multicolumn{1}{c|}{0.1375034}& \multicolumn{1}{c|}{0.00768532}                 \\
				\multicolumn{1}{c|}{} & \multicolumn{1}{c|}{1}  &0.04670408 & 0.09401281 & \multicolumn{1}{c|}{0.1764996} &\multicolumn{1}{c|}{0.01327334}  & \multicolumn{1}{c|}{0.1501311}              & \multicolumn{1}{c|}{0.01701934}                               \\
				
				\multicolumn{1}{c|}{}  & \multicolumn{1}{c|}{1.5}   & 0.2177388  &0.3566961 & \multicolumn{1}{c|}{0.576717}  &\multicolumn{1}{c|}{0.02193555}  & \multicolumn{1}{c|}{0.1891187}              & \multicolumn{1}{c|}{0.0278058}                        \\  
				
				\multicolumn{1}{c|}{} & \multicolumn{1}{c|}{2}  & 0.5154611 & 0.7899285 & \multicolumn{1}{c|}{1.210099} & \multicolumn{1}{c|}{0.03416945}   & \multicolumn{1}{c|}{0.2417847}              & \multicolumn{1}{c|}{0.04120925}                               \\
\cline{1-8} 
				\multicolumn{1}{c|}{0.5}&\multicolumn{1}{c|}{0.5}    & 0.01375736 &0.008212465 & \multicolumn{1}{c|}{0.003358186} & \multicolumn{1}{c|}{0.0081828} & 
				\multicolumn{1}{c|}{0.1375034}& \multicolumn{1}{c|}{0.00768532}                 \\
				\multicolumn{1}{c|}{}&\multicolumn{1}{c|}{1}     & 0.004917317 & 0.01222534 & \multicolumn{1}{c|}{0.02793089}&
				\multicolumn{1}{c|}{0.01327334}  & \multicolumn{1}{c|}{0.1501311}              & \multicolumn{1}{c|}{0.01701934}                               \\
				\multicolumn{1}{c|}{}&\multicolumn{1}{c|}{1.5}   & 0.05338891 & 0.08794855 & \multicolumn{1}{c|}{0.1443359} &\multicolumn{1}{c|}{0.02193555}  & \multicolumn{1}{c|}{0.1891187}              & \multicolumn{1}{c|}{0.0278058}                           \\
				\multicolumn{1}{c|}{}&\multicolumn{1}{c|}{2}     & 0.1591721 & 0.23588 & \multicolumn{1}{c|}{0.3536742} & \multicolumn{1}{c|}{0.03416945}   & \multicolumn{1}{c|}{0.2417847}              & \multicolumn{1}{c|}{0.04120925}                               \\
\cline{1-8} 
				\multicolumn{1}{c|}{1}&\multicolumn{1}{c|}{0.5}     & 0.02334328 & 0.0186799 & \multicolumn{1}{c|}{0.01343518} & \multicolumn{1}{c|}{0.0081828} & 
				\multicolumn{1}{c|}{0.1375034}& \multicolumn{1}{c|}{0.00768532}                 \\	\multicolumn{1}{c|}{}&\multicolumn{1}{c|}{1}      & 0.002585846 & 0.00202712 & \multicolumn{1}{c|}{0.003527147} &\multicolumn{1}{c|}{0.01327334}  & \multicolumn{1}{c|}{0.1501311}              & \multicolumn{1}{c|}{0.01701934}                                                                                            \\
				\multicolumn{1}{c|}{}& \multicolumn{1}{c|}{1.5}       &0.0145548  & 0.02439929 & \multicolumn{1}{c|}{0.04097613} &\multicolumn{1}{c|}{0.02193555}  & \multicolumn{1}{c|}{0.1891187}              & \multicolumn{1}{c|}{0.0278058}               \\
				\multicolumn{1}{c|}{}& \multicolumn{1}{c|}{2}      & 0.05925014 & 0.08613785 & \multicolumn{1}{c|}{0.1266097}& \multicolumn{1}{c|}{0.03416945}   & \multicolumn{1}{c|}{0.2417847}              & \multicolumn{1}{c|}{0.04120925}                               \\
\cline{1-8} 
				\multicolumn{1}{c|}{2} &\multicolumn{1}{c|}{0.5}   & 0.03507632 & 0.03235116  & \multicolumn{1}{c|}{0.02909741} & 
				\multicolumn{1}{c|}{0.0081828} & 
				\multicolumn{1}{c|}{0.1375034}& \multicolumn{1}{c|}{0.00768532}                 \\
				\multicolumn{1}{c|}{} &\multicolumn{1}{c|}{1}   & 0.01217588 & 0.00972373 & \multicolumn{1}{c|}{0.007239784} &\multicolumn{1}{c|}{0.01327334}  & \multicolumn{1}{c|}{0.1501311}              & \multicolumn{1}{c|}{0.01701934}                               \\
				\multicolumn{1}{c|}{}&\multicolumn{1}{c|}{1.5}    &0.004400968 &0.004131617  & \multicolumn{1}{c|}{0.00474087}&
				\multicolumn{1}{c|}{0.02193555}  & \multicolumn{1}{c|}{0.1891187}              & \multicolumn{1}{c|}{0.0278058}               \\
				\multicolumn{1}{c|}{}&\multicolumn{1}{c|}{2}    & 0.01175158
				&0.01574687 & \multicolumn{1}{c|}{0.02203485} & \multicolumn{1}{c|}{0.03416945}   & 
				\multicolumn{1}{c|}{0.2417847}              & \multicolumn{1}{c|}{0.04120925}                               \\
\cline{1-8} 
		\end{tabular}}
	\end{table}
\vfill\eject
By Corollary \ref{cor2}(ii) we infer that
\begin{equation}
\label{eq:Zn}
Z_n:=\frac1{\widehat\gamma_n\sqrt2}\left(\frac{1+2\rho}{1+\rho}\right)^{1/2}\sqrt{k_n}\log\frac n{k_n}
\big\{\widehat\gamma_n- \gamma\big(1+ t_n^{-1}\big)\big\}
\stackrel{\mathcal{D}}{\longrightarrow} N(0,1).
\end{equation}
Asymptotic confidence intervals for $\gamma$ can be constructed using either \eqref{eq:norm1} or \eqref{eq:Zn}.
In the second simulation study we investigated how fast the distribution result \eqref{eq:Zn} kicks in.
We simulated the quantity $Z_n$ 5000 times.
According to condition ($B_2$), we used $k_n$ values less than $\log^2n$. First, we investigated the Fr\'echet distribution with shape parameter $1/\gamma$ that belongs to the Hall
model with parameters $D_1=1$, $D_2=-\gamma/2$ and $\beta=1$.
The simulation was done for $\gamma=1$, $\rho=1$, $p=1$, $n=900$ and $k_n=10$.
We found empirically that $n=900$ is the threshold sample size to obtain a good normal approximation in \eqref{eq:Zn}.
Figure \ref{fig:h1} contains the histogram of the simulated $Z_n$ quantities and the fitted normal curve with estimated parameters.
The mean of the simulated $Z_n$ values is -0.06, 
the simulated
standard deviation is 0.8974. 
The mean of the simulated $\widehat\gamma_n$ values is 1.1116. 
The bias of the mean is in accordance with the bias term $\gamma t_n^{-1}$ in \eqref{eq:Zn}.
Due to the biased estimator in the leading factor $1/(\widehat\gamma_n\sqrt2)$ of $Z_n$, the simulated standard deviation of $Z_n$ is smaller than the asymptotic value 1. We performed the chi-square test for normality, and we
obtained the p-value 0.2965. 

\begin{figure}[H]
\centering
\includegraphics[trim = 0 0 0 50, clip]{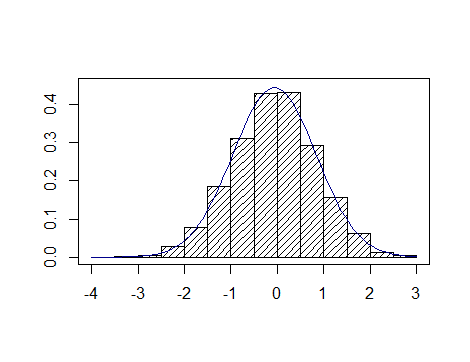}
\vspace*{-1.2cm}
\caption{}
\label{fig:h1}
\end{figure}

\goodbreak
We investigated two more distributions from the Hall model:
Case 1: $\gamma=1$, $D_1 =1$ and $D_2 = 1/2$, $\beta = 3/4$;
Case 2: $\gamma=2$, $D_1 =1$ and $D_2 = 1$, $\beta = 1$.
We used $\rho=3$, $p=2$, $n=500$ and $k_n=7$ for Case 1, and $\rho=1$, $p=1$, $n=900$ and $k_n=10$ for Case 2.
These $n$ values are the threshold sample sizes to obtain a good normal approximation in \eqref{eq:Zn}.
We obtained the following numerical results. Case 1: mean of the simulated $Z_n$ values: 0.0013, 
standard deviation of the $Z_n$ values: 0.9127, 
mean of the simulated $\widehat\gamma_n$ values: 1.0667; 
Case 2: mean of the simulated $Z_n$ values: -0.0393, 
standard deviation of the $Z_n$ values: 0.8878, 
mean of the simulated $\widehat\gamma_n$ values: 2.2267. 
The p-value of the chi-square test for normality is 0.323 
for Case 1, and 0.6428 
for Case 2.
Figures \ref{fig:h2} and \ref{fig:h3} contain the histograms of the simulated quantities and the fitted normal curves for Case 2 and Case 3, respectively.

\begin{figure}[H]
\centering
\includegraphics[trim = 0 0 0 50, clip]{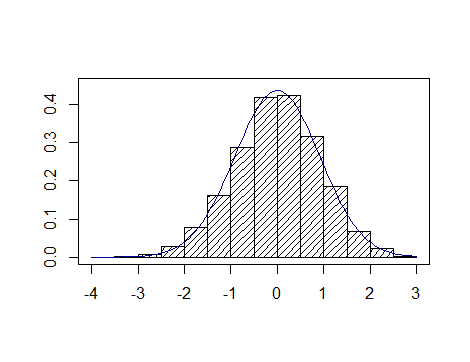}
\vspace*{-1.2cm}
\caption{}
\label{fig:h2}
\end{figure}

\begin{figure}[H]
\centering
\includegraphics[trim = 0 0 0 50, clip]{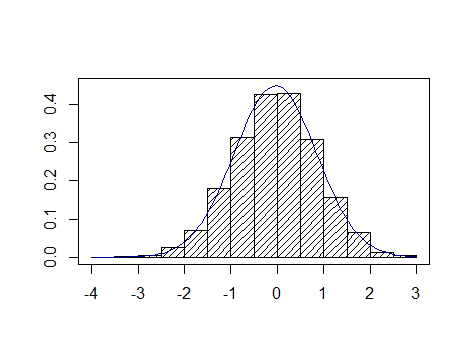}
\vspace*{-1.2cm}
\caption{}
\label{fig:h3}
\end{figure}

\section{Proofs}
Let $G$ be a distribution function on $\mathbb R$ and assume the existence of constants $d_n>0$ and $c_n$ such that
\begin{equation}
\label{eq:maxdomain}
\lim_{n\to\infty} G^n(d_n x + c_n) = G_c(x)
\end{equation}
for all $x\in\mathbb R$ with some nondegenerate distribution function $G_c(x)$ necessarily being an extreme value distribution function
\[
G_c(x)=\exp\big(-(1+cx)^{-1/c}\big),
\]
where $c\in \mathbb R$ and $x$ is such that $1+cx>0$ and $(1+cx)^{-1/c}$ is interpreted as $e^{-x}$ if $c=0$.
Whenewer \eqref{eq:maxdomain} holds we say that that $G$ belongs to the maximum domain of attraction of $G_c$ and we write $G\in\Delta(c)$.
Set $U(s) := -G^\leftarrow(1-s)$, $0 \le s <1$, where the arrow means the inverse function.
From \cite[equation (1.12)]{CsHM} we know the following statement.
\begin{proposition}
\label{prop:maxdomain}
$G\in\Delta(c)$ if and only if
\[
\lim_{s\downarrow0} \frac{U(xs)-U(ys)}{U(vs)-U(ws)} = \frac{x^{-c}-y^{-c}}{v^{-c}-w^{-c}},
\]
where for $c=0$ the limit is understood as $(\log x-\log y)/(\log v-\log w)$.
\end{proposition}
Let $RV^\infty_\alpha$ ($RV^0_\alpha$) denote the class of regularly varying functions at infinity (zero) with index $\alpha$.
\begin{lemma}
\label{lemma1}
Assume the conditions of Theorem \ref{thm1}. Then the distribution function $H(\cdot):=(-K(1-\cdot))^\leftarrow$ satisfies $H\in\Delta (-\rho)$.
\end{lemma}

\begin{proof}
A simple calculation yields
$K(t)=-\int_{Q(1/2)}^{Q(1-t)}J_1(u)du$, where
\[
J_1(u)=pJ (1- F(u))(\log u)^{p-1}u^{-1} \in RV^\infty_{-(\rho/\gamma)-1}.
\]

If $\rho>0$ then $K(t)=\int_{Q(1-t)}^\infty J_1(u)du+c$, where $c$ is a constant, and by Karamata's theorem (see e.g. \cite[Theorem 1.5.11]{BGT}) we obtain 
\[
K(t)=\frac\gamma\rho Q(1-t) J_1(Q(1-t))(1+o(1))+c\quad (t\to0).
\]
Similarly, if $\rho<0$ then
\[
K(t)=\frac\gamma\rho Q(1-t) J_1(Q(1-t))(1+o(1))\quad (t\to0).
\]
Theorem 1.5.12 of \cite{BGT} implies that
\begin{equation}
\label{FQ}
1-F(Q(1-t)) \sim t\quad (t\to0).
\end{equation}
Then using \eqref{FQ} and $\log Q(1-t) \sim -\gamma\log t\ (t\to0)$, we have
\[
Q(1-t) J_1(Q(1-t)) \sim p(-\gamma\log t)^{p-1}J(t)\quad (t\to0).
\]
Hence, if $\rho>0$ then
\begin{equation}
\label{eq:K1}
K(t)=t^\rho\hat L(t)(1+o(1))+c \quad (t\to0),
\end{equation}
and if $\rho<0$ then
\begin{equation}
\label{eq:K2}
K(t)=-t^\rho\hat L(t)(1+o(1))\quad (t\to0),
\end{equation}
where
\begin{equation}
\label{eq:hatL}
\hat L(t)=\frac{p\gamma^p}{|\rho|} (-\log t)^{p-1}\bar\ell(t)\in RV^0_0.
\end{equation}
Equations \eqref{eq:K1}, \eqref{eq:K2} and \eqref{eq:hatL} imply that for $\rho \ne 0$,
\begin{equation}
\label{eq:Klimit1}
\lim_{s\downarrow0} \frac{K(xs)-K(ys)}{K(vs)-K(ws)} = \frac{x^{\rho}-y^{\rho}}{v^{\rho}-w^{\rho}}.
\end{equation}
If $\rho=0$, then for distinct values $0<x,y<\infty$,
\begin{equation}
\label{eq:K}
K(xs) - K(ys) = \bar\ell(\xi)(g(xs)-g(ys))
\end{equation}
where $\xi$ is between $xs$ and $ys$. Since $\bar\ell$ is slowly varying,
we have
\begin{equation}
\bar\ell(\xi) \sim \bar\ell(s)\quad (s\downarrow0).
\end{equation}
Moreover, by Lagrange's mean value theorem, with some $\eta$ between $\log Q(1-(xs)-)$ and $\log Q(1-(ys)-)$,
\begin{equation}
g(xs)-g(ys) = p\eta^{p-1}(\log Q(1-(ys)-) - \log Q(1-(xs)-)).
\end{equation}
Using \eqref{eq:quant} and the fact that $\log Q(1-s-)$ is slowly varying at zero, we have
\begin{equation}
\label{eq:logQ}
\eta\sim \log Q(1-s-) \sim -\gamma\log s \quad (s\downarrow0),
\end{equation}
and
\begin{equation}
\label{eq:Q}
\log Q(1-(ys)-) - \log Q(1-(xs)-) \to \gamma \log(x/y)\quad (s\downarrow0).
\end{equation}
By \eqref{eq:K}-\eqref{eq:Q} it follows that
\begin{equation}
\label{eq:Kdiff}
K(xs) - K(ys) \sim p \gamma^p \log(x/y) (-\log s)^{p-1} \bar\ell(s) \quad (s\downarrow0).
\end{equation}
Therefore,
\begin{equation}
\label{eq:Klimit2}
\lim_{s\downarrow0}
\frac{K(xs) - K(ys)}{K(vs) - K(ws)}
= \frac{\log x - \log y}{\log v - \log w}
\end{equation}
for all distinct $0< x,y,v,w <\infty$.
Equations \eqref{eq:Klimit1}, \eqref{eq:Klimit2} and Proposition \ref{prop:maxdomain} imply the statement of the lemma.
\end{proof}

Choose any sequence of positive constants $\delta_n$ such that $n\delta_n<n$ and $n\delta_n\to0$ as $n\to\infty$.
The following two sequences of functions govern the asymptotic behavior of $S_n(p)$:
\[
\psi_n(x)=\psi_{n,K}(x)=
\begin{cases}
\frac{k_n^{1/2}\left\{K\left(\frac{k_n}n+x\frac{k_n^{1/2}}n\right)-K\left(\frac{k_n} n\right)\right\}}{n^{1/2}a_n} & \text{if}\quad -\frac{k_n^{1/2}}2\le x\le \frac{k_n^{1/2}}2,\\
\psi_n\left(-\frac{k_n^{1/2}}2\right) & \text{if}\quad -\infty<x<-\frac{k_n^{1/2}}2,\\
\psi_n\left(\frac{k_n^{1/2}}2\right) & \text{if}\quad \frac{k_n^{1/2}}2<x<\infty,
\end{cases}
\]
and
\[
\varphi_n(y)=\varphi_{n,K}(y)=
\begin{cases}
\frac{K(y/n)-K(1/n)}{n^{1/2}a_n} & \text{if}\quad 0<y\le n-n\delta_n,\\
\frac{K(1-\delta_n)-K(1/n)}{n^{1/2}a_n} & \text{if}\quad n-n\delta_n<y<\infty.
\end{cases}
\]

\begin{lemma}
\label{lemma2}
Assume the conditions of Theorem \ref{thm1}. Then $\psi_n(x), \varphi_n(y)\to0$, $x\in\mathbb R$, $y>0$.
\end{lemma}

\begin{proof}
The statement is a consequence of Lemmas 2.11 and 2.12 of \cite{CsHM} and Lemma \ref{lemma1} above. 
\end{proof}

\begin{proof}[Proof of Proposition \ref{propseq}]
If $\rho>-1/2,\ \rho\ne0$ then by Lemma 2.9 of \cite{CsHM} and by Lemma \ref{lemma1} above we have
\[
\sigma(1/n,k_n/n)\sim \left(\frac{2\rho^2}{(1+\rho)(1+2\rho)}\right)^{1/2}\left(\frac{k_n}n\right)^{\rho+1/2}\hat L(k_n/n),
\]
which is the same as \eqref{sigma}. If $\rho=0$, then by (2.29) of \cite{CsHM} and by Lemma \ref{lemma1}
we have $\sigma(1/n,k_n/n)\sim\sigma(0,k_n/n)$, and using Lemma 2.10 of \cite{CsHM}, we obtain
\[
\lim_{s\downarrow0}\sqrt{s}(K(\lambda s)-K(s))/\sigma(0,s)=2^{-1/2}\log \lambda \quad \text{for all } 0<\lambda<\infty.
\]
Then by \eqref{eq:Kdiff},
\[
\sigma(0,s) \sim p\gamma^p\sqrt{2s}(-\log s)^{p-1}\bar\ell(s) \quad (s\downarrow0),
\]
which implies the statement for $\rho=0$.

Statement $\mu_n\sim\gamma^p\alpha_n$ follows from the facts $-J(\cdot)g(\cdot) \in RV_\rho^0$, $\log Q(1-s-) \sim -\gamma \log s$ and from Karamata's theorem.
\end{proof}

\begin{proof}[Proof of Theorem \ref{thm1}]
The Corollary of \cite{Viharos93} and Lemma \ref{lemma2} imply statement (i).
To prove statement (ii) write
\begin{equation}
\label{eq:decomp1}
\bar\mu_n = \mu_n - n \int^{1/n}_0 J(u)(-g(u)) du + d_{n,n}(-g(1/n))
=: \mu_n - r_n^{(1)} + r_n^{(2)}.
\end{equation}
We have to prove that
\begin{equation}
\label{eq:r}
\frac{r_n^{(1)}}{\sqrt n a_n}\to 0\quad \text{and}\quad \frac{r_n^{(2)}}{\sqrt n a_n}\to 0.
\end{equation}
By Karamata's theorem, \eqref{eq:logQ} and Proposition \ref{propseq}, with some constant $c$ we have
\begin{equation*}
\begin{split}
\frac{r_n^{(1)}}{\sqrt n a_n}
& \sim c\,\frac{ n^{-\rho}\bar\ell( 1/n )(\log n)^p }
{\sqrt n (k_n/n)^{\rho+1/2}(\log(n/k_n))^{p-1}\bar\ell( k_n/n )}\\
&=c\frac{(\log n)^p\bar\ell( 1/n )}
{k_n^{\rho+1/2}(\log(n/k_n))^{p-1}\bar\ell( k_n/n )}.
\end{split}
\end{equation*}
By the Potter bounds (\cite[Theorem 1.5.6]{BGT}), for any $A>1$ and $\delta>0$, there exist $N$ such that
\[
\frac{\bar\ell( 1/n )}{\bar\ell( k_n/n )} \le A k_n^\delta \quad
\text{and}\quad \frac{\log n}{\log(n/k_n)} \le A k_n^\delta\quad \text{for any}\quad n\ge N. 
\]
We choose $\delta>0$ such that $p\delta < \rho -\varepsilon +1/2$. It follows that with some constant $c_1$,
\[
\frac{r_n^{(1)}}{\sqrt n a_n} \le c_1 \frac{\log n}{k_n^{\rho-p\delta+1/2}}
\le c_1 \frac{\log n}{k_n^\varepsilon}
\]
if $n\ge N$. A similar upper bound for $r_n^{(2)}/(\sqrt n a_n)$ implies \eqref{eq:r}.
\end{proof}

\begin{proof}[Proof of Corollary \ref{cor1}]
Using Theorem \ref{thm1} and Proposition \ref{propseq}, we obtain
\begin{equation}
\label{eq:norm2}
\beta_n\left(\frac{S_n(p)}{\alpha_n}-\frac{\mu_n}{\alpha_n}\right) \stackrel{\mathcal{D}}{\longrightarrow} N(0,1),
\end{equation}
where
\begin{equation}
\label{eq:beta}
\beta_n = \frac1{\gamma^p p\sqrt2}\left(\frac{1+2\rho}{1+\rho}\right)^{1/2}\sqrt{k_n}\log\frac n{k_n}.
\end{equation}
Since $\mu_n/\alpha_n \to \gamma^p$ and $\beta_n \to \infty$, we have $S_n(p)/\alpha_n \stackrel{\mathbb{P}}{\longrightarrow} \gamma^p$.

By Lagrange's mean value theorem
\[
\left(\frac{S_n(p)}{\alpha_n}\right)^{1/p}-\left(\frac{\mu_n}{\alpha_n}\right)^{1/p}
= \frac 1p \xi^{(1/p)-1} \left(\frac{S_n(p)}{\alpha_n}-\frac{\mu_n}{\alpha_n}\right)
\]
with some $\xi$ between $\mu_n/\alpha_n$ and $S_n(p)/\alpha_n$.
Therefore,
\[
\beta_n\left(\left(\frac{S_n(p)}{\alpha_n}\right)^{1/p}-\left(\frac{\mu_n}{\alpha_n}\right)^{1/p}\right)
\stackrel{\mathcal{D}}{\longrightarrow} \frac 1p \gamma^{1-p}N(0,1).
\]
\end{proof}

\goodbreak
\begin{proof}[Proof of Corollary \ref{cor2}]
Proof of (i). 
To treat  $\bar\mu_n$, we use the decomposition \eqref{eq:decomp1}.
For $\mu_n$ we obtain
\begin{equation}
	\label{eq:decomp2}
\begin{split}
\mu_n & = n\int^{k_n/n}_0J(u)(\log u^{-\gamma})^p\left(1+\frac{\log\ell(u)}{\log u^{-\gamma}}\right)^p du\\
& = n\int^{k_n/n}_0J(u)(\log u^{-\gamma})^p du \\
& \quad + n\int^{k_n/n}_0J(u)(\log u^{-\gamma})^p \left[\left(1+\frac{\log\ell(u)}{\log u^{-\gamma}}\right)^p-1\right] du \\
& =: \mu_n^{(1)} + \mu_n^{(2)}.
\end{split}
\end{equation}
By Karamata's theorem,
\begin{equation*}
\int^x_0J(u)(\log u^{-\gamma})^p du \sim \frac1{\rho+1} x J(x)(\log x^{-\gamma})^p\quad \text{as } x\to 0.
\end{equation*}
Therefore, using Condition ($B_1$), we have 
\[
\begin{split}
\sqrt{k_n} & \log\frac n{k_n} \frac{|\mu_n^{(2)}|}{\alpha_n} \\
& \le \frac{n\sqrt{k_n}\log\frac n{k_n}}{\alpha_n}\int_0^{k_n/n}J(u) (\log u^{-\gamma})^p du
\sup_{0\le u \le k_n/n}\left| \left(1+\frac{\log\ell(u)}{\log u^{-\gamma}}\right)^p-1\right| \\
& \sim 
\gamma^p \sqrt{k_n}\log\frac n{k_n}\sup_{0\le u \le k_n/n}\left| \left(1+\frac{\log\ell(u)}{\log u^{-\gamma}}\right)^p-1\right|.
\end{split}
\]
By 
\begin{equation}
\label{eq:pp}
(1+x)^p = 1 + px + O(x^2)\quad \text{as } x\to0
\end{equation}
and condition ($B_1$) it follows that
\begin{equation}
\label{eq:mu2}
\sqrt{k_n}\log\frac n{k_n} \frac{|\mu_n^{(2)}|}{\alpha_n} \to 0.
\end{equation}
For the first term we obtain 
\[
\mu_n^{(1)} = \frac{n\gamma^p}{(\rho+1)^{p+1}} \int_{(\rho+1)\log(n/k_n)}^\infty t^pe^{-t}dt
= \frac{n\gamma^p}{(\rho+1)^{p+1}} \Gamma(p+1,(\rho+1)\log(n/k_n)),
\]
where
\[
\Gamma(a,x) = \int_x^\infty t^{a-1} e^{-t} dt
\]
is the incomplete gamma function.
It is known that
\[
\Gamma(a,x) = x^{a-1} e^{-x} \left(\sum_{j=0}^{n-1} b_jx^{-j} + M_n(x)\right),
\]
where $b_j = (a-1)(a-2)\cdots(a-j)$ and
\begin{equation}
\label{eq:Mn}
M_n(x) = O(x^{-n})\quad \text{as } x\to\infty
\end{equation}
(see equation (2.02) in \cite{Olver}). Recall the notation $t_n=(\rho+1)\log(n/k_n)$. Then
\begin{equation}
	\label{eq:mu1}
\frac{\mu_n^{(1)}}{\alpha_n} = \gamma^p (1 + pt_n^{-1} + M_2(t_n) ).
\end{equation}
For $r_n^{(2)}$ in \eqref{eq:decomp1} we obtain
\[
	r_n^{(2)} = \frac{1}{n^\rho(\rho+1)}(\log Q(1-(1/n)-))^p
	\sim \frac{1}{n^\rho(\rho+1)} (\gamma \log n)^p,
\]
implying that
\begin{equation}
\sqrt{k_n}\log\frac n{k_n} \frac{|r_n^{(2)}|}{\alpha_n}
\sim \gamma^p \frac{(\log n)^p}{k_n^{\rho+\frac12}(\log(n/k_n))^{p-1}}.
\end{equation}
Condition ($B_2$) implies $\log(n/k_n) \sim \log n$. Therefore, by Condition ($B_3$) we have
\begin{equation}
	\label{eq:rn2}
	\sqrt{k_n}\log\frac n{k_n} \frac{|r_n^{(2)}|}{\alpha_n} \to 0.
\end{equation}
A similar argument yields that
\begin{equation}
	\label{eq:rn1}
	\sqrt{k_n}\log\frac n{k_n} \frac{|r_n^{(1)}|}{\alpha_n} \to 0
\end{equation}
(cf.~the proof of Theorem \ref{thm1}(ii)).
Recall \eqref{eq:beta}. Using the decompositions \eqref{eq:decomp1} and \eqref{eq:decomp2}, equations \eqref{eq:mu2}, \eqref{eq:mu1}, \eqref{eq:rn2} and \eqref{eq:rn1}, we obtain
\begin{equation}
\label{eq:appr}
\beta_n\left(\frac{S_n(p)}{\alpha_n}-\frac{\bar\mu_n}{\alpha_n}\right)
= \beta_n\left(\frac{S_n(p)}{\alpha_n} - \gamma^p (1 + pt_n^{-1} + M_2(t_n) )\right)
+ o(1).
\end{equation}
Theorem \ref{thm1}(i), condition ($B_2$), \eqref{eq:Mn} and \eqref{eq:appr} imply 
\[
\beta_n \left(\frac{S_n(p)}{\alpha_n}- \gamma^p \left(1 + pt_n^{-1} \right) \right)
\stackrel{\mathcal{D}}{\longrightarrow} N(0,1).
\]
This completes the proof of part (i).

Proof of (ii).
Using the same argument as in the proof of Corollary \ref{cor1}, we have 
\[
\beta_n \bigg(\left(\frac{S_n(p)}{\alpha_n}\right)^{1/p} - \gamma\left(1 + pt_n^{-1}\right)^{1/p} \bigg)
\stackrel{\mathcal{D}}{\longrightarrow} \frac 1p \gamma^{1-p} N(0,1).
\]
Applying \eqref{eq:pp} with $1/p$ replacing $p$, we obtain
\[
\left(1 + \frac p{t_n}\right)^{1/p}
= 1+t_n^{-1} + O(t_n^{-2}).
\]
Therefore, by condition ($B_2$)
\[
\beta_n \bigg(\left(\frac{S_n(p)}{\alpha_n}\right)^{1/p} - \gamma\big(1 + t_n^{-1}\big) \bigg)
\stackrel{\mathcal{D}}{\longrightarrow} \frac 1p \gamma^{1-p} N(0,1).
\]
This completes the proof of part (ii).
\end{proof}

\bigskip\noindent
{\bf Acknowledgement.}
We thank the referee for the valuable remarks and suggestions that helped us to improve the paper.
We also thank P\'eter Kevei for his helpful advice.
This research was supported by grant TUDFO/47138-1/2019-ITM of the Ministry for Innovation and Technology, Hungary.

\bibliographystyle{abbrv}
\bibliography{power-extremes}

\end{document}